\documentclass[letterpaper, 10 pt, conference]{ieeeconf}  

\IEEEoverridecommandlockouts                              
\usepackage{lipsum}
\usepackage{amsfonts}
\usepackage{epstopdf}
\usepackage{amssymb}  
\usepackage{amsmath}
\usepackage{float}
\usepackage{algorithm}
\usepackage{dsfont}

\usepackage{tikz}
\usetikzlibrary{patterns}
\usetikzlibrary{arrows}
\tikzset{cross/.style={path picture={
  \draw[black]
    (path picture bounding box.south east)--(path picture bounding box.north west)
    (path picture bounding box.south west)--(path picture bounding box.north east);
}}}

\newcommand{\mb}[1]{\mathbb{#1}}
\newcommand{\mc}[1]{\mathcal{#1}}

\newtheorem{definition}{\textsc{Definition}}
\newtheorem{cond}{\textsc{Condition}}
\newtheorem{theorem}{\textsc{Theorem}}[section]

\newtheorem{proposition}{\textsc{Proposition}}
\newtheorem{remark}{\textsc{Remark}}[section]

\title{\LARGE \bf
Minimal time problem for discrete crowd models with a  localized vector field
}

\author{Michel Duprez \and Morgan Morancey \and Francesco Rossi
\thanks{This work has been carried out in the framework of Archim\`ede Labex (ANR-11-LABX-0033) and of the A*MIDEX project (ANR-11-IDEX-0001-02), funded by the ``Investissements d'Avenir" French Government programme managed by the French National Research Agency (ANR). 
}
\thanks{M.~Duprez and M.~Morancey are with Aix Marseille Universit\'e, CNRS, Centrale Marseille, I2M,  Marseille, France
		{\tt\small mduprez@math.cnrs.fr} {\tt\small morgan.morancey@univ-amu.fr}}
\thanks{F.~Rossi is with Dipartimento di Matematica ``Tullio Levi-Civita", Universit\`{a} degli Studi di Padova, Via Trieste 63, 35121 Padova, Italy
		{\tt \small francesco.rossi@math.unipd.it}}
}

\begin{document}

\maketitle
\thispagestyle{empty}
\pagestyle{empty}

\begin{abstract}
In this work, we study the  minimal time to steer a given crowd to a desired
configuration. The control is a vector field, 
representing a perturbation of the crowd velocity, 
localized on a fixed control set. 

\noindent We characterize the minimal time for a discrete crowd model, both for exact and approximate controllability.  
This leads to an algorithm that computes the control and the minimal time. 
We finally present a numerical simulation.
\end{abstract}

\section{Introduction}
In recent years, the study of systems describing a crowd of interacting agents 
has drawn a great interest from the control community.
A better understanding of such interaction phenomena can have a strong impact in several key applications, such as road traffic and egress problems for pedestrians. For few reviews about this topic, see 
\textit{e.g.} \cite{axelrod,active1,camazine,CPTbook,helbing,jackson,SepulchreReview}.

Beside the description of interactions, it is now relevant to study problems of control of crowds, \textit{i.e.} of controlling such systems by acting on few agents, or on a small subset of the configuration space. The nature of the control problem relies on the model used to describe the crowd. In this article, we focus on discrete  models, in which  the position of each agent is clearly identified; the crowd dynamics is described by a large dimensional ordinary differential equation, in which couplings of terms represent interactions. For control of such models, a large literature is available, see \textit{e.g.} reviews \cite{bullo,kumar,lin}, as well as applications, both to pedestrian crowds \cite{ferscha,luh} and to road traffic \cite{canudas,hegyi}.

The key aspect of such crowd models, is that agents are considered identical, or indistinguishable. Thus, control problems need to take into account that each configuration is indeed defined modulo a permutation of agents. Since in general the number of agents is large, it is then interesting to find methods in which control goals (controllability, optimal control) are reached without computing all the permutations.

In the present work, we study the following discrete  model,
where the crowd is described by a vector with $nd$ components ($n,d\in\mb{N}^*$) representing the positions of $n$ agents in 
the space $\mb{R}^d$. The natural (uncontrolled) vector field is denoted by  
$v: \mb{R}^d\rightarrow\mb{R}^d$, assumed Lipschitz and uniformly bounded.
We act on the vector field in a fixed subdomain  $\omega$ of the space, 
which will be a  {nonempty open convex subset} of $\mb{R}^d$. 
The admissible controls are thus functions of the form $\mathds{1}_{\omega}u:\mb{R}^d\times\mb{R}^+\rightarrow\mb{R}^d$.
The dynamics is given by the following 
 ordinary differential equation 
\begin{equation}\label{eq ODE}
\left\{\begin{array}{l}
\dot x_i(t) =v(x_i(t)) + \mathds{1}_{\omega}(x_i(t)) u(x_i(t),t),\\
x_i(0)=x^0_i.
\end{array} \right.
\end{equation}
where $X^0:=\{x^0_1,...,x^0_n\}\subset\mathbb{R}^d$ is the initial configuration of the crowd.
This representation with configurations can be applied only if the different 
agents are considered identical or interchangeable, as it is often the case for crowd models with a large number of agents.
The function $v+\mathds{1}_{\omega}u$ represents the velocity vector field acting on the crowd $X:=\{x_1,...,x_n\}$.
 Thus we can modify this vector field only on a given nonempty open subset $\omega$ of the space $\mb{R}^d$. This kind of control is one of the originality of our research. Such constraint is highly non-trivial, since the control problem is non-linear. 
 At the best of our knowledge, minimal time problems in this setting have not been studied.
 
Notice that~\eqref{eq ODE} represents a specific crowd model, as the velocity field $v$ is given, and interactions between agents are not taken into account. Nevertheless, it is necessary to understand control properties for such simple equations as a first step, before dealing with vector fields depending on the crowd itself. 
Moreover, one can consider this problem as the local perturbation of an interaction model along a reference trajectory described by $v$.

The first question about control of ~\eqref{eq ODE} is to describe controllability results, i..e which configurations can be steered from one to another. We solved this problem in \cite{DMR17}, whose main results are recalled in Section \ref{Sec:MainResults}.

When controllability is ensured, it is then interesting to study minimal time problems. Indeed, from the theoretical point of view, it is the first problem in which optimality conditions can be naturally defined. More related to applications described above, minimal time problems play a crucial role: egress problems can be described in this setting, while traffic control is often described in terms of minimization of (maximal or average) total travel time. 

For discrete models, the dynamics can be written in terms of finite-dimensional control systems. For this reason, minimal time problems can sometimes be addressed with classical (linear or non-linear) control theory, see \textit{e.g.} \cite{agrabook,jurdjevic,sontag}. 
Our main aim here is to derive a method that takes into account the indistinguishability of agents, without passing through the computation of all possible permutations. 
Classical methods are then not adapted. 
 For this reason, our main results presented in Section \ref{Sec:MainResults} will explicitly identify fast algorithms to find minimizing permutations. Moreover, these efficient methods will be also useful for numerical methods, presented in Section \ref{sec:num sim}.

\begin{remark} Another relevant approach fo crowds modeling is given by continuous models. There, the idea is to represent the crowd by the spatial density of agents; in this setting, the evolution of the density solves a partial differential equation of transport type.  
Nonlocal terms (such as convolutions) model the interactions between the agents. For the few available results of control of such systems, see e.g. \cite{PRT15,CPRT17,CPRT17b,DMR17,DMR18}.
\end{remark}

This paper is organised as follows. 
In Sec.~\ref{Sec:MainResults}, we give the setting and our main results about the minimal time for (exact and approximate) controllability for~\eqref{eq ODE}. These results are proved in Sec.~\ref{sec:opt time finite dim}.
Finally, in Sec.~\ref{sec:num sim} we introduce an algorithm to compute the infimum time for approximate control of discrete models and give a numerical example.

\section{Main results}
\label{Sec:MainResults}
To ensure the well-posedness of System~\eqref{eq ODE}, we search a control  $\mathds{1}_{\omega}u$
satisfying the following condition:
\begin{cond}[Carath\'eodory condition]Let $\mathds{1}_{\omega}u$ be such that 
for all $t\in\mathbb{R}$, $x\mapsto \mathds{1}_{\omega}u(x,t)$ is Lipschitz, 
for all $x\in\mathbb{R}^d$, $t\mapsto \mathds{1}_{\omega}u(x,t)$ is measurable
and there exists $M>0$ such that $\|\mathds{1}_{\omega}u\|_{\infty}\leqslant M$.
\end{cond}
In this setting, System~\eqref{eq ODE} is well defined. Then, the {\it flow} can be properly defined.
\begin{definition}
We define the flow associated to a vector field $w:\mb{R}^d\times\mb{R}^+\rightarrow\mb{R}^d$
satisfying the Carath\'eodory condition
as the application $(x^0,t)\mapsto\Phi_t^w(x^0)$ such that, for all $x^0\in\mb{R}^d$, 
$t\mapsto\Phi_t^w(x^0)$ is the unique solution to 
\begin{equation*}
\left\{\begin{array}{l}
\dot x(t) =w(x(t),t)\mbox{ for a.e. }t\geqslant 0,\\
x(0)=x^0.
\end{array}\right.
\end{equation*}
\end{definition}

One of the key properties of solutions to System~\eqref{eq ODE} is that they cannot separate or merge particles. Thus, the general interesting settings for crowd models is the one of distinct configurations as defined below.
\begin{definition}
A configuration $X=\{x_1,...,x_n\}$ is said to be disjoint if $x_i\neq x_j$ for all $i\neq j$.
\end{definition}
Since we deal with velocities $v+\mathds{1}_{\omega}u$ satisfying the Carath\'eo\-dory condition, if $X^0$ is a disjoint configuration,
then the solution $X(t)$ to System~\eqref{eq ODE} is also a disjoint configuration at each time $t\geqslant 0$.

From now on, we will assume that the following condition is satisfied by initial and final configurations.
\begin{cond}[Geometric condition]\label{cond1}
Let  $X^0, X^1$ be two disjoint configurations in $\mb{R}^d$ satisfying:
\begin{enumerate}
\item[(i)] For each  $i \in \{1, \dots, n\}$, 
there exists $t^0>0$ such that $\Phi_{t^0}^v(x^0_i)\in \omega$.
\item[(ii)] For each $i \in \{1, \dots, n\}$,  
there exists $t^1>0$ such that $\Phi_{-t^1}^{v}(x^1_i)\in \omega$.
\end{enumerate}
\end{cond}
The Geometric Condition \ref{cond1} means that the trajectory of each particle crosses the control region forward in time and the trajectories of each position of the target configuration crosses the control region backward in time.
It is the minimal condition that we can expect to steer any initial condition to 
any target. Indeed, we proved in \cite{DMR17}  that one can approximately steer an initial to a final configuration of the System ~\eqref{eq ODE} if they satisfy the Geometric Condition \ref{cond1}.

In the sequel, we will define the following functions 
for all $x\in\mb{R}^d$ and $j\in\{0,1\}$:
\begin{equation*}
\left\{\begin{array}{l}
\overline{t}^0(x):=\inf\{t\in\mb{R}^+:\Phi_{t}^v(x)\in\overline{\omega}\},\\
\overline{t}^1(x):=\inf\{t\in\mb{R}^+:\Phi_{-t}^v(x)\in\overline{\omega}\},\\\noalign{\smallskip}
t^0(x):=\inf\{t\in\mb{R}^+:\Phi_{t}^v(x)\in\omega\},\\\noalign{\smallskip}
t^1(x):=\inf\{t\in\mb{R}^+:\Phi_{-t}^v(x)\in\omega\}.
\end{array}\right.
\end{equation*}

{It is clear that it always holds $\overline{t}^j(x)\leqslant t^j(x)$.
In some situations, this inequality can be strict. For example, in Figure \ref{fig:ex CE CA},
it holds $\overline{t}^1(x^1_1)<t^1(x^1_1)$. Moreover, in this specific case these functions can even be discontinous with respect to $x$.

\begin{figure}[h]
\begin{center}
\begin{tikzpicture}[scale=0.7]
\fill[pattern=dots,opacity = 0.5] (3,0) -- (3,-1) -- (-3,-1) -- (-3,0) -- cycle;
\draw (3,0) -- (3,-1) -- (-3,-1) -- (-3,0) -- cycle;
\node[cross,thick,minimum size=4pt] at (-2,1) {};
\node[cross,thick,minimum size=4pt] at (2,-2) {};
\path (-1.5,1) node {$x^0_1$};
\path (1.7,-2.5) node {$x^1_1$};
\path (2,-0.5) node {$\omega$};
\draw (-2,1) -- (-2,1);
\path (-2.3,-1.5) node {$v$};
\draw (-2,-2) -- (-2,1);
\draw  (0,-2) arc (0:-180:1);
\draw  (0,-2) arc (0:-180:-1);
\path (-2.3,0.5) node {$v$};
\draw (-2,0.5) -- (-2.1,0.6);
\draw (-2,0.5) -- (-1.9,0.6);
\draw (-2,-1.5) -- (-2.1,-1.4);
\draw (-2,-1.5) -- (-1.9,-1.4);
\draw[dotted,thick]  (2.2,-2) arc (0:90:1.2);
\draw[dotted,thick] (-2,-0.2) -- (1,-0.8);
\node[cross,thick,minimum size=4pt] at (2.2,-2) {};
\path (2.5,-2.5) node {$\widetilde{x}^1_1$};
\end{tikzpicture}\end{center}
\caption{Example of difference between $t^1(x^1_1)$ and $\overline{t}^1(x^1_1)$. }\label{fig:ex CE CA}
\end{figure}
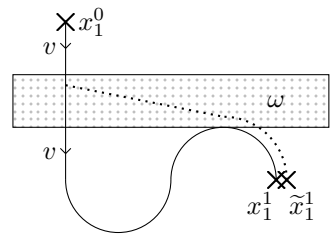

For simplicity, we use the notations
\begin{equation}\label{def:t^l_i}
 \overline{t}^j_i:=\overline{t}^j(x_i^j) \mbox{ and } t^j_i:=t^j(x_i^j),
\end{equation}
for $i\in\{1,...,n\}$ and $j \in \{0,1\}$. We then define 
$$
\left\{\begin{array}{l}
M^*_e(X^0,X^1):=\max\{t_i^0,t_i^1:i=1,...,n\},\\ 
M^*_a(X^0,X^1):=\max\{t_i^0,\overline{t}_i^1:i=1,...,n\}.
\end{array}\right.$$

We now state our first main result.
\begin{theorem}
\label{th:discret exact}
Let $X^0:=\{x^0_1,...,x^0_n\}$ and $X^1:=\{x^1_1,...,x^1_n\}$ be disjoint configurations satisfying the Geometric Condition \ref{cond1}.
Arrange the sequences 
$\{t^0_i\}_i$ and $\{t^1_j\}_j$ to be increasingly and decreasingly ordered, respectively.
Then 
\begin{equation}\label{OT disc CE}
M_e(X^0,X^1):=\max_{i\in\{1,...,n\}}|t^0_i+t^1_i|
\end{equation}
is the infimum time $T_e(X^0,X^1)$ for  exact control of System \eqref{eq ODE} in the following sense:
\begin{itemize}
\item[(i)] For each $T>M_e(X^0,X^1)$, System \eqref{eq ODE} is exactly controllable  from $X^0$ to $X^1$ at time $T$, \textit{i.e.}
there exists a control  $\mathds{1}_{\omega}u:\mb{R}^d\times\mb{R}^+\rightarrow\mb{R}^d$ satisfying the Carath\'eodory condition  and steering $X^0$ exactly to $X^1$.
\item[(ii)] For each $T\in (M^*_e(X^0,X^1),M_e(X^0,X^1)]$,  System \eqref{eq ODE} is not exactly controllable  from $X^0$ to $X^1$.
\item[(iii)] There exists (at most) a finite number of times $T\in[0,M^*_e(X^0,X^1)]$ for 
which System \eqref{eq ODE} is exactly controllable  from $X^0$ to $X^1$.
\end{itemize}

\end{theorem}

We give a proof of Theorem \ref{th:discret exact} in Section \ref{sec:opt time finite dim}.

We now turn to approximate controllability.
We will use the following distance between configurations.
\begin{definition}
Consider  $X^0:=\{x^0_1,...,x^0_n\}$ and $X^1:=\{x^1_1,...,x^1_n\}$ two configurations of $\mb{R}^d$ and define
the distance 
 \begin{equation*}
\|X^0 -X^1\|:=\inf_{\sigma\in \mathbb{S}_n}\left(\sum_{i=1}^n\frac{1}{n} |x^0_i-x^1_{\sigma(i)}| \right),
\end{equation*}
where $\mathbb{S}_n$ is the set of permutations on $\{1,...,n\}$.
\end{definition}
This distance\footnote{This distance coincides with the Wasserstein distance for empirical measures (see \cite[p. 5]{V03}).} clearly takes into account the indistinguishability of agents, in the sense that its value does not depend on the ordering of $X_0,X_1$.

We now state our second main result.
\begin{theorem}
\label{th:discret approx}
Let $X^0:=\{x^0_1,...,x^0_n\}$ and $X^1:=\{x^1_1,...,x^1_n\}$ be  disjoint configurations satisfying the Geometric Condition \ref{cond1}.
Arrange the sequences 
$\{t^0_i\}_i$ and $\{\overline{t}^1_j\}_j$ to
 be increasingly and decreasingly ordered, respectively.
Then  
\begin{equation*}
M_a(X^0,X^1):=\max_{i\in\{1,...,n\}}|t^0_i+\overline{t}^1_i|
\end{equation*}
is the infimum time $T_a(X^0,X^1)$ for approximate  controllability of System \eqref{eq ODE} in the following sense:
\begin{itemize}
\item[(i)] For each $T>M_a(X^0,X^1)$, System \eqref{eq ODE} is approximately controllable  from $X^0$ to $X^1$ at time $T$, \textit{i.e.}
for any $\varepsilon >0$, there exists a control  $\mathds{1}_{\omega}u$ satisfying the Carath\'eodory condition such that the associated solution $X(t)$ to System \eqref{eq ODE} satisfies 
$\|X(T)-X_1\| < \varepsilon.$
 \item[(ii)] For each $T\in (M^*_a(X^0,X^1),M_a(X^0,X^1)]$,  System \eqref{eq ODE} is not approximately controllable  from $X^0$ to $X^1$.
\item[(iii)] There exists (at most) a finite number of times $T\in[0,M^*_a(X^0,X^1)]$ for which System 
\eqref{eq ODE} is approximately controllable  from $X^0$ to $X^1$.
\end{itemize}
\end{theorem}
We give a proof of Theorem \ref{th:discret approx} in Section \ref{sec:opt time finite dim}.

In both theorems, controllability can occur at small times but it is a very specific situation which is not entirely due to the control. See Remark~\ref{rmq:T2*} for examples.

\begin{remark}
It is well know that the notions of approximate and exact controllability 
are equivalent for finite dimensional linear systems, when the control acts linearly, see e.g. \cite{C09}.
We remark that it is not the case for System \eqref{eq ODE}, which highlights the fact that we are dealing with
a non-linear control problem. The difference is indeed related to the fact that for exact and approximate controllability, tangent trajectories  give different behaviors.
 For example, in Figure \ref{fig:ex CE CA}, if we denote by $X^0:=\{x^0_1\}$ and $X^1:=\{x^1_1\}$, 
 then it holds   $M_a(X^0,X^1)< M_e(X^0,X^1)$ due to the presence of a tangent trajectory.
An approximate trajectory is represented as dashed lines in the case $T\in (  M_a(X^0,X^1), M_e(X^0,X^1))$
in Figure \ref{fig:ex CE CA}.
 \end{remark}

\section{Proofs of main results}\label{sec:opt time finite dim}
In this section, we prove Theorem \ref{th:discret exact} and \ref{th:discret approx}

\subsection{Minimal time for exact controllability}

We first obtain the following result:

\begin{proposition}\label{prop: dim finie}
Let $X^0:=\{x^0_1,...,x^0_n\}\subset\mb{R}^d$ and $X^1:=\{x^1_1,...,x^1_n\}\subset\mb{R}^d$ be two disjoint configurations satisfying the Geometric Condition \ref{cond1}.
Consider the sequences $\{t_i^0\}_{i}$ and $\{t_i^1\}_{i}$ given in \eqref{def:t^l_i}. 
Then 
\begin{equation}\label{minimal time}
\widetilde{M}_e(X^0,X^1):=\min_{\sigma\in \mathbb{S}_n}\max_{i\in\{1,...,n\}}|t^0_i+t_{\sigma (i)}^1|
\end{equation}
is the infimum time $T_e(X^0,X^1)$ to  exactly control System \eqref{eq ODE} in the sense of Theorem \ref{th:discret exact}.
\end{proposition}

\begin{proof} 
We first prove the result corresponding to { Item (i)} of Theorem \ref{th:discret exact}. Let $T:=\widetilde{M}_e(X^0,X^1)+\delta$ with $\delta>0$. For all $i\in\{1,...,n\}$, there exist $s_i^0\in (t_i^0,t_i^0+\delta/3)$ and $s_i^1\in (t_i^1,t_i^1+\delta/3)$
such that $y_i^0:=\Phi_{s_i^0}^v(x_i^0)\in\omega\mbox{ and }y_i^1:=\Phi_{-s_i^1}^v(x_i^1)\in\omega$.

\textbf{Item (i), Step 1:} 
The goal is to build  a flow with no intersection of the trajectories $x_i(t),x_j(t)$ with $i\neq j$. 
For all $i,j\in\{1,...,n\}$, we define the cost
\begin{equation*}
K_{ij}(y_i^0,s_i^0,y_j^1,s_j^1):= \|(y_i^0,s_i^0)-(y_j^1,T-s_j^1)\|_{\mb{R}^{d+1}}
\end{equation*}
if $s_i^0<T-s_j^1$ and $K_{ij}(y_i^0,s_i^0,y_j^1,s_j^1):=\infty$ otherwise.
Consider the minimization problem:
\begin{equation}\label{eq:inf}
\inf\limits_{\pi\in\mc{B}_n}\frac{1}{n}\sum_{i,j=1}^n K_{ij}(y_i^0,s_i^0,y_j^1,s_j^1)\pi_{ij},
\end{equation}
where $\mc{B}_n$ is the set of the bistochastic $n\times n$ matrices, \textit{i.e.} 
the matrices $\pi:=(\pi_{ij})_{1\leqslant i,j\leqslant n}$ satisfying, for all $i,j\in\{1,...,n\}$,
$\sum_{i=1}^n\pi_{ij}=1$, $\sum_{j=1}^n\pi_{ij}=1$, $\pi_{ij}\geqslant 0.$
The infimum in \eqref{eq:inf} is finite since $T>\widetilde{M}_e(X^0,X^1)$.
The problem \eqref{eq:inf} is a linear minimization problem on the closed convex set $\mc{B}_n$.
Hence, as a consequence of Krein-Milman's Theorem (see \cite{KM40}), the functional \eqref{eq:inf} admits a minimum
at an extremal point of $\mathcal{B}_n$,
\textit{i.e.} a permutation matrix. 
Let $\sigma$ be a permutation, for which the associated matrix minimizes \eqref{eq:inf}.
Consider the straight 
 trajectories  $y_i(t)$ steering $y_i^0$ at time $s_i^0$ to $y_{\sigma(i)}^1$ at time $T-s_{\sigma(i)}^1$, that are explicitly defined by
\begin{equation}
y_i(t):=
\frac{T-s_{\sigma(i)}^1-t}{T-s_{\sigma(i)}^1-s_i^0}y_i^0+\frac{t-s_{i}^0}{T-s_{\sigma(i)}^1-s_i^0}y_{\sigma(i)}^1.
\label{e:straight}
\end{equation}
We now prove by contradiction that these trajectories have no intersection:
Assume that there exist $i$ and $j$ such that the associated trajectories $y_i(t)$ and $y_j(t)$ intersect. 
If we associate $y^0_i$ and $y^0_j$ to $y^0_{\sigma(j)}$ and 
$y^0_{\sigma(i)}$ respectively, \textit{i.e.} we consider the permutation 
 $\sigma\circ\mc{T}_{i,j}$,  where $\mc{T}_{i,j}$ is the transposition between the  $i^{\mbox{\footnotesize th}}$ 
and the $j^{\mbox{\footnotesize th}}$ elements,
then the associated  cost \eqref{eq:inf} is strictly smaller than 
the cost associated to $\sigma$
(see Figure \ref{fig:geo}). This is in contradiction with the fact that $\sigma$ minimizes \eqref{eq:inf}.
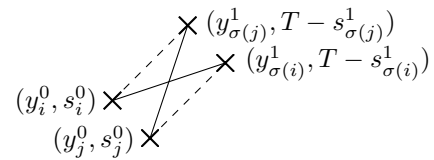
\begin{figure}[h]
\begin{center}
\begin{tikzpicture}[scale=0.5]
\node[cross,thick,minimum size=4pt] at (1,0) {};
\node[cross,thick,minimum size=4pt] at (0,1) {};
\node[cross,thick,minimum size=4pt] at (3,2) {};
\node[cross,thick,minimum size=4pt] at (2,3) {};
\draw[-] (0,1) -- (3,2);
\draw[-] (1,0) -- (2,3);
\draw[dashed] (0,1) -- (2,3);
\draw[dashed] (1,0) -- (3,2);
\path (-1.5,1) node {$(y_i^0,s_i^0)$};
\path (-0.5,-0.1) node {$(y_j^0,s_j^0)$};
\path (5,3) node {$(y_{\sigma(j)}^1,T-s_{\sigma(j)}^1)$};
\path (6,2) node {$(y_{\sigma(i)}^1,T-s_{\sigma(i)}^1)$};
\end{tikzpicture}
\caption{An optimal permutation.}\label{fig:geo}
\end{center}\end{figure}

\textbf{Item (i), Step 2: } 
We now define a corresponding control sending $x_i^0$
to $x_{\sigma(i)}^1$ for all $i\in\{1,...,n\}$.
Consider a trajectory $z_i$ satisfying:
\begin{equation*}
z_i(t):=\left\{\begin{array}{ll}
\Phi_t^v(x_i^0)&\mbox{ for all }t\in(0,s_i^0),\\
y_i(t)&\mbox{ for all }t\in(s_i^0,T-s_{\sigma(i)}^1),\\
\Phi_{t-T}^v(x_i^1)&\mbox{ for all }t\in(T-s_{\sigma(i)}^1,T).
\end{array}\right.
\end{equation*}
The trajectories $z_i$ have no intersection.
Since $\omega$ is convex, then, using the definition of the trajectory $y_i(t)$ in \eqref{e:straight},
the points $y_i(t)$ belong to $\omega$ for all $t\in(s_i^0,T-s^1_{\sigma(i)})$.
For all $i\in \{1,...,n\}$, 
choose $r_i,R_i$ satisfying $0<r_i<R_i$ and such that for all $t\in(s_i^0,T-s^1_{\sigma(i)})$ it holds
\begin{equation*}
B_{r_i}(z_i(t))\subset B_{R_i}(z_i(t))\subset\omega
\end{equation*}
and, for all $t\in(0,T)$ and $i,j\in\{1,...,n\}$, it holds
\begin{equation*}
B_{R_i}(z_i(t))\cap B_{R_j}(z_j(t))=\varnothing.
\end{equation*}
Such radii $r_i,R_i$ exist as a consequence of the fact that we deal with a finite number of trajectories that do not cross.
The corresponding control can be chosen as a $\mc{C}^{\infty}$ function satisfying
\begin{equation*}
u(x,t):=\left\{\begin{array}{ll}
\dfrac{y_{\sigma(i)}^1-y_i^0}{T-s_{\sigma(i)}^1-s_i^0}-v&
\hspace*{-2mm}\begin{array}{l}\mbox{if }t\in (s_i^0,T-s^1_{\sigma(i)})\\
\mbox{ and }x\in B_{r_i}(z_i(t)),\end{array}\\ 
u(x,t):=0 
&\hspace*{-2mm}\begin{array}{l}\mbox{if }t\in (s_i^0,T-s^1_{\sigma(i)})\\
\mbox{ and }x\not\in B_{R_i}(z_i(t)),\end{array}\\
u(x,t):=0 &\mbox{if }t\not\in (s_i^0,T-s^1_{\sigma(i)}).\end{array}\right.
\end{equation*}
This control then satisfies the Carath\'eodory condition and each $i$-th component of the associated solution to System \eqref{eq ODE}  is $z_i(t)$, thus $u$ steers $x_i^0$ to $x^1_{\sigma(i)}$ in time $T$.

\textbf{Item (ii):} Assume that System \eqref{eq ODE} 
is exactly controllable at a time $T> M^*_e(X^0,X^1)$, and consider $\sigma$ the corresponding permutation defined by $x_{i}(T)=x^1_{\sigma(i)}$. The idea of the proof is that the trajectory steers $x^0_i$ to $\omega$ in time $t^0_i$, then it moves inside $\omega$ for a small but positive time, then it steers a point from $\omega$ to $x^i_{\sigma(i)}$ 
in time $t^1_{\sigma(i)}$, hence $T>t^0_i+t^1_{\sigma(i)}$.

Fix an index $i\in\{1,...,n\}$. 
First recall the definition of $t^0_i,t^1_{\sigma(i)}$ and observe that it holds both $T>t^0_i$ and $T>t^1_{\sigma(i)}$. Then, the trajectory $x_i(t)$ satisfies\footnote{These estimates hold even if $x^0_i\in \omega$, for which it holds $t^0_i=0$.}  $x_i(t)\not\in \omega$ for all $t\in (0,t^0_i)$, as well as $x_i(t)\not\in\omega$ for all $t\in(T-t^1_{\sigma(i)},T)$. Moreover, we prove that it exists $\tau_i\in(0,T)$ for which it holds $x_i(\tau_i)\in\omega$. By contradiction, if such $\tau_i$ does not exist, then the trajectory $x_i(t)$ never crosses the control region, hence it coincides with $\Phi_t^v(x^0_i)$. But in this case, by definition of $t^0_i$ as the infimum of times such that $\Phi_t^v(x^0_i)\in\omega$ and recalling that $t^0_i<T$, there exists $\tau_i\in(t^0_i,T)$ such that it holds $x_i(\tau_i)=\Phi_{\tau_i}^v(x^0_i)\in\omega$. Contradiction. Also observe that $\omega$ is open, hence there exists $\epsilon_i$ such that $x_i(\tau)\in\omega$ for all $\tau\in(\tau_i-\epsilon_i,\tau_i+\epsilon)$.

We merge the conditions $x_i(t)\not\in \omega$ for all $t\in (0,t^0_i)\cup(T-t^1_{\sigma(i)},T)$ with $x_i(\tau)\in\omega$ for all $\tau\in(\tau_i-\epsilon_i,\tau_i+\epsilon_i)$ with a given $\tau_i\in(0,T)$. This implies that it holds $t^0_i<\tau_i<T-t^1_{\sigma(i)}$, hence \begin{equation*}T>t^0_i+t_{\sigma (i)}^1.
\end{equation*}

Such estimate holds for any $i\in\{1,...,n\}$. Thus, using the definition of $\widetilde{M}_e(X^0,X^1)$, it holds $T>\widetilde{M}_e(X^0,X^1)$.

\textbf{Item (iii):} 
By definition of $M_e^*(X^0,X^1)$, 
there exists $l\in\{0,1\}$ and $m\in \{1,...,n\}$ such that $M^*_e(X^0,X^1)=t_m^l$. 
We only study the case $l=0$, since the case $l=1$ can be recovered by reversing time. 
By definition of $t_m^0$, the trajectory $\Phi^v_t(x^0_m)$ satisfies $\Phi^v_t(x^0_m)\not\in \omega$ for all $t\in[0,M^*_e(X^0,X^1)]$. Then, for any choice of the control $u$ localized in $\omega$, it holds $\Phi^{v+\mathds{1}_{\omega}u}_{t}(x_m^0)=\Phi^v_t(x^0_m)$, \textit{i.e.} the choice of the control plays no role in the trajectory starting from $x^0_m$ on the time interval $t\in[0,M^*_e(X^0,X^1)]$. Observe that it holds $v(\Phi^v_t(x^0_m))\neq 0$ for all $t\in[0,M_e^*(X^0,X^1)]$, due to the fact that the vector field is time-independent and the trajectory $\Phi^v_t(x^0_m)$ enters $\omega$ for some $t>M_e^*(X^0,X^1)$. 

We now prove that the set of times $t\in[0,M_e^*(X^0,X^1)]$ for which exact controllability holds is finite. A necessary condition to have exact controllability at time $t$ is that the equation $\Phi^v_t(x^0_m)=x^1_i$ admits a solution for some time $t\in[0,M_e^*(X^0,X^1)]$ and index $i\in\{1,...,n\}$. Then, we aim to prove that the set of times-indexes $(t,i)$ solving such equation is finite. By contradiction, assume to have an infinite number of solutions $(t,i)$. Since the set $i\in\{1,...,n\}$ is finite, this implies that there exists an index $I$ and an infinite number of (distinct) times $t_k\in[0,M_e^*(X^0,X^1)]$ such that $\Phi^v_{t_k}(x^0_m)=x^1_I$.
By compactness of $[0,M_e^*(X^0,X^1)]$, there exists a converging subsequence (that we do not relabel) $t_k\to t_*\in[0,M_e^*(X^0,X^1)]$. Since $v$ is continuous, we can compute $v(\Phi^v_{t_*}(x^0_m))$ by using the definition and taking the subsequence $t_k\to t_*$, that gives 
$$v(\Phi^v_{t_*}(x^0_m))=\lim_{k\to\infty}\frac{\Phi^v_{t_k}(x^0_m)-\Phi^v_{t_*}(x^0_m)}{t_k-t^*}
=0.$$ This is in contradiction with the fact that $v(\Phi^v_t(x^0_m))\neq 0$ for all $t\in[0,M_e^*(X^0,X^1)]$.
\end{proof}

Formula \eqref{minimal time} leads to the proof of Theorem \ref{th:discret exact}.

{\it Proof of Theorem \ref{th:discret exact}}. Consider $\widetilde{M}_e(X^0,X^1)$ given in \eqref{minimal time}.
By relabeling particles, we assume that the sequence  $\{t_i^0\}_{i\in\{1,...,n\}}$ is increasingly ordered.
Let $\sigma_0$ be a minimizing permutation in  \eqref{minimal time}.
We build recursively a sequence of permutations $\{\sigma_1,...,\sigma_n\}$ as follows: 
Let $k_1$ be such that $t_{\sigma_0(k_1)}^1$ is a maximum of the set $\{t_{\sigma_0(1)}^1,...,t_{\sigma_0(n)}^1\}$.
We denote by $\sigma_1:=\sigma_0\circ\mc{T}_{1,k_1}$,  where $\mc{T}_{i,j}$ is the transposition between the  $i$-th
and the $j$-th elements. 
It holds 
$$
t^0_{k_1}+t^1_{\sigma_0(k_1)}\geqslant \max\{t_1^0+t^1_{\sigma_0(1)},t_1^0+t^1_{\sigma_1({1})}, t_{k_1}^0+t^1_{\sigma_1(k_1)}\}.
$$
Thus $\sigma_1$ minimizes  \eqref{minimal time} too, since it holds
$$
\max_{i\in\{1,...,n\}}\{t_i^0+t_{\sigma_0(i)}^1\}
\geqslant \max_{i\in\{1,...,n\}}\{t_i^0+t_{\sigma_1(i)}^1\}.
$$
We then build iteratively the permutation $\sigma_k$.
The sequence $\{t^1_{\sigma_n(1)},...,t^1_{\sigma_n(n)}\}$ is then decreasing and $\sigma_n$ is a minimizing permutation in \eqref{minimal time}.
Thus $\widetilde{M}_e(X^0,X^1)=M_e(X^0,X^1)$.
\hfill $\blacksquare$ 

With Theorem \ref{th:discret exact}, we give an explicit and simple expression of the infimum time for exact controllability of discrete models. This result is also useful for numerical simulations of Section \ref{sec:num sim}.

\subsection{Minimal time for approximate controllability}

We now prove Theorem \ref{th:discret approx}, which characterizes the infimum time for approximate control of System \eqref{eq ODE}.

{\it Proof of Theorem \ref{th:discret approx}.} We first prove \textbf{Item (i)}. 
As for Theorem~\ref{th:discret exact}, we first prove that the minimal time is 
\begin{equation*}
\widetilde{M}_a(X^0,X^1):=\min_{\sigma\in \mathbb{S}_n}\max_{i\in\{1,...,n\}}|t^0_i+\overline{t}_{\sigma (i)}^1|.
\end{equation*}
Indeed, as in the proof of Theorem~\ref{th:discret exact}, the permutation method implies
$\widetilde{M}_a(X^0,X^1)=M_a(X^0,X^1)$. This point is left to the reader.

First assume that $T>\widetilde{M}_a(X^0,X^1)$. Let $\varepsilon>0$. 
For each $x^1_i$, we prove the existence of points $y_i^1$ satisfying 
\begin{equation}
|y_i^1-x^1_i|\leqslant \varepsilon
\mbox{ and }y_i:=\Phi^v_{-\overline{t}_i^1}(y_i^1)\in\omega.\label{e-yi1}
\end{equation}
For each $x^1_i$, observe that the Geometric Condition \ref{cond1} implies that either $x^1_i\in\omega$ or that the trajectory enters $\omega$ backward in time. In the first case, define $y_i^1:=x_i^1$. In the second case, remark that $v(\Phi^v_{-t}(x^1_i))$ is nonzero for a whole interval $t\in[0,\tilde t]$, with $\tilde t>\bar t^1_i$, 
and $\Phi^v_{-\bar t_i^1}(x^1_i)\in\overline{\omega}$, hence the flow $\Phi^v_{-\bar t^1_i}(\cdot)$ is a diffeomorphism in a neighborhood of $x^1_i$. Then, there exists $y_i^1\in\mb{R}^d$ such that \eqref{e-yi1} is satisfied.

We denote by $Y^1:=\{y^1_1,...,y_n^1\}$.
For all $i\in\{1,...,n\}$, since $y_i\in\omega$, then $t^1(y_i^1)\leqslant \overline{t}_i^1$, hence 
$$\widetilde{M}_e(X^0,Y^1)\leqslant\widetilde{M}_a(X^0,X^1)<T.$$
Proposition \ref{prop: dim finie} implies that we can exactly steer $X^0$ to $Y^1$ at time $T$ with a control $u$
satisfying the Carath\'eodory condition. 
Denote by $X(t)$ the solution to System \eqref{eq ODE} for the initial condition $X^0$ and the control $u$.
It then holds 
\begin{equation*}\|X^1-X(T)\|=\|X^1-Y^1\|\leqslant \sum_{i=1}^n\frac{1}{n}|y_i^1-x^1_i|\leqslant\varepsilon.\label{e-apprW1}
\end{equation*}

We now prove \textbf{Item (ii)}. Consider a control time $T>M^*_a(X^0,X^1)$ at which System \eqref{eq ODE} is approximately controllable. We aim to prove that it satisfies $T>M_a(X^0,X^1)$. 
For each $k\in\mathbb{N}^*$, there exists a control $u_k$
satisfying the Carath\'eodory condition such that the corresponding solution $X_k(t)$ to System \eqref{eq ODE} satisfies
\begin{equation}\label{eq:W1 1/k}
\|X^1-X_k(T)\|\leqslant 1/k.
\end{equation}
We denote by $Y^1_k:=\{y_{k,1}^1,...,y_{k,n}^1\}$ the configuration defined by $y_{k,i}^1:=X_{k,i}(T)$, where $X_{k,i}$ is the $i$-th component of $X_k$. Since $X^0$ is disjoint and $u_k$ satisfies the Carath\'eodory condition, then $Y^1_k$ is disjoint too.
We now prove that it holds
\begin{equation} \label{T_Me*}
T > M^*_e(X^0, Y^1_k).
\end{equation}
Since  $T>M^*_a(X^0,X^1)$, then \eqref{T_Me*} is equivalent to
$T > t_i^1(y_{k,i}^1)$ for all $i\in\{1,...,n\}$.
By contradiction, assume that there exists $j \in \{1,\dots,n\}$ such that $t^1(y^1_{k,j}) \geqslant T$.
Assume that $t^1(y^1_{k,j}) > T$, 
the case $t^1(y^1_{k,j}) = T$ being similar since $\omega$ is open. 
Then for any $t \in [0,T]$ it holds $\Phi_{-t}^v(y^1_{k,j}) \not \in \omega$. Thus, the localized control does not act on the trajectory, \textit{i.e.} for each $t\in[0,T]$ it holds $\Phi_{-t}^v(y^1_{k,j}) = \Phi_{-t}^{v+\mathds{1}_{\omega}u_k}(y^1_{k,j})$.

Since $y_{k,j}^1=\Phi_{T}^{v+\mathds{1}_{\omega}u_k}(x^0_j)=\Phi^v_T(x^0_j)$, then
$\Phi_{t}^{v}(x^0_j)\not\in\omega$
for all $t\in[0,T]$.
This is a contradiction with the fact that  $t^0_j\leqslant M_a^*(X^0,X^1) < T$. Thus~\eqref{T_Me*} holds.
Since $Y_k^1=X_k(T)$, then Proposition \ref{prop: dim finie} implies that 
\begin{equation}\label{ine Me Ma}
T>\widetilde{M}_e(X^0,Y^1_k).
\end{equation}
For each control $u_k$, denote by $\sigma_k$ the permutation for which it holds $y^1_{k,i}=\Phi_{T}^{v+\mathds{1}_{\omega}u_k}(x^0_{\sigma_k(i)})$. Up to extract a subsequence, for all $k$ large enough, $\sigma_k$ is equal to a permutation $\sigma$.
Inequality \eqref{eq:W1 1/k} implies that for all $i\in\{1,...,n\}$ it holds
\begin{equation}\label{y_k,i}
y_{k,i}^1\underset{k\rightarrow\infty}{\longrightarrow}x_{\sigma(i)}^1.
\end{equation}
Since $t^1(y_{k,i}^1)\leqslant \widetilde{M}_e(X^0,Y^1_k)<T$, up to a subsequence, for a $s_i\geqslant 0$, it holds
\begin{equation}\label{t_k,i}
t^1(y_{k,i}^1)\underset{k\rightarrow\infty}{\longrightarrow}s_{i}.
\end{equation}
Using \eqref{y_k,i}, \eqref{t_k,i}  and the continuity of the flow, it holds
$$\begin{array}{l}
|\Phi_{-t^1(y_{k,i}^1)}^v(y_{k,i}^1)-\Phi_{-s_i}^v(x_{\sigma(i)}^1)|
\underset{k\rightarrow\infty}{\longrightarrow}0.\end{array}$$
The fact that $\Phi_{-t^1(y_{k,i}^1)}^v(y_{k,i}^1)\in\overline{\omega}$ for each $i=1,\ldots,n$ leads to $\Phi_{-s_i}^v(x_{\sigma(i)}^1)\in\overline{\omega}$.
Thus
$\overline{t}^1(x_{\sigma(i)}^1)\leqslant \lim\limits_{k\rightarrow\infty} t^1(y^1_{k,i})$.
Denoting $\delta:=(T-\widetilde{M}_e(X^0,X^1))/2$,
using \eqref{ine Me Ma}, we obtain
\begin{equation*}\begin{array}{rcl}
\widetilde{M}_a(X^0,X^1)
&\leqslant&\max_{i\in\{1,...,n\}}|t^0_i+\overline{t}_{\sigma (i)}^1|\\
&\leqslant&\max_{i\in\{1,...,n\}}|t^0_i+t^1(y^1_{k,\sigma(i)})|+\delta\\
&=&\widetilde{M}_e(X^0,Y^1_k)+\delta<T.
\end{array}
\end{equation*}

We finally prove \textbf{Item (iii)} of Theorem~\ref{th:discret approx}. 
Let $T \in (0,M_a^*(X^0,X^1))$ be such that System~\eqref{eq ODE} is approximately controllable. 
For any $\varepsilon>0$, there exists $u_\varepsilon$ such that the associated trajectory to System~\eqref{eq ODE} satisfies
\begin{equation}\label{W1 Ma*}
\|X_{\varepsilon}(T)-X^1\| < \varepsilon.
\end{equation}
There exists $j \in \{1, \dots,n\}$ such that it holds 
$t^0(x^0_j) = M_a^*(X^0,X^1) > T$ or $\overline{t}^1(x^1_j) = M_a^*(X^0,X^1) > T$. 
Assume that $t^0(x^0_j) = M_a^*(X^0,X^1) > T$, the case $\overline{t}^1(x^1_j) = M_a^*(X^0,X^1)$ being similar. 
Define $x_{\varepsilon,j}(t):= \Phi_t^{v+\mathds{1}_{\omega}u_{\varepsilon}}(x_j^0)$. 
 Inequality \eqref{W1 Ma*} implies that it exists $k \in \{1, \dots, n\}$ such that
\begin{equation} \label{proche_cible}
| x_{\varepsilon,j}(T) - x^1_{k_{\varepsilon}} | < \varepsilon.
\end{equation}
As $t^0(x^0_j) > T$, the trajectory $ \Phi_t^{v}(x_j^0)$ does not cross the control set $\omega$ for $t\in[0,T)$, hence 
$$x_{\varepsilon,j}(T)= \Phi_T^{v+\mathds{1}_{\omega}u_{\varepsilon}}(x_j^0)=\Phi_T^v(x_j^0)$$ does not depend on $\varepsilon$.
Define $R:= \frac{1}{2} \min_{p,q} |x^1_p - x^1_q |$, that is strictly positive since $X^1$ is disjoint. 
For each $\epsilon<R$, estimate~\eqref{proche_cible} gives $k_{\varepsilon}=k$ independent on $\varepsilon$ and
$x_{\varepsilon,j}(T) =  \Phi_t^{v}(x_j^0)=x_k^1$.
Use now the proof of Item (iii) in Proposition \ref{prop: dim finie} to prove that the equation $\Phi_t^{v}(x_j^0)=x_k^1$ admits a finite number of solutions $(t,k)$ with $t\in[0,t^0(x^0_j)]$ and $k\in \{1, \dots,n\}$.
\hfill $\blacksquare$

\begin{remark}\label{rmq:T2*}
We  illustrate Item (iii) with two examples.
\begin{itemize}
\item[$\bullet$] Figure~\ref{fig:ex (0,T*)} (left). 
The vector field $v$ is $(1,0)$, thus uncontrolled trajectories are right translations. 
 The time $M_e^*(X^0,X^1)$ 
at which we can act on the particles and the minimal time $M_e(X^0,X^1)$ 
are respectively equal
 to $1$ and $2$.
We observe that System \eqref{eq ODE} is neither exactly controllable nor approximately controllable
on the whole interval $[0,M_e^*(X^0,X^1))$.
\item[$\bullet$] Figure~\ref{fig:ex (0,T*)} (right). 
The vector field $v$ is $(-y,x)$, thus uncontrolled trajectories are rotations with constant angular velocity. 
 The time $M_e^*(X^0,X^1)$ at which we can act on the particles and the minimal time $M_e(X^0,X^1)$ 
are respectively equal
to $3\pi/4$ and $\pi$.
We remark that System \eqref{eq ODE} is exactly controllable, then approximately controllable, 
at time $T=\pi/2\in [0,M_e^*(X^0,X^1))$.
\end{itemize}

\begin{figure}[ht]
\begin{center}
\begin{tikzpicture}[scale=0.75]
\fill[pattern=dots,opacity = 0.5] (-2,-1.5) -- (0,-1.5) -- (0,1.5) -- (-2,1.5) -- cycle;
\draw (-2,-1.5) -- (0,-1.5) -- (0,1.5) -- (-2,1.5) -- cycle;
\node[cross,thick,minimum size=4pt] at (-3,0) {};
\node[cross,thick,minimum size=4pt] at (1,0) {};
\path (-3,0.5) node {$x^0_1$};
\path (1,0.5) node {$x^1_1$};
\path (-1,-1) node {$\omega$};
\path (-2.5,-0.3) node {$v$};
\draw (-3,0) -- (1,0);
\draw (-2.4,0) -- (-2.5,0.1);
\draw (-2.4,0) -- (-2.5,-0.1);
\draw (0.5,1) -- (1.5,1);
\draw (0.5,1.1) -- (0.5,0.9);
\draw (1.5,1.1) -- (1.5,0.9);
\path (1,1.25) node {\scriptsize 1};
\end{tikzpicture}
\hspace{1cm}
\begin{tikzpicture}[scale=0.75]
\fill[pattern=dots,opacity = 0.5] (-2,-1.5) -- (0,-1.5) -- (0,1.5) -- (-2,1.5) -- cycle;
\draw (-2,-1.5) -- (0,-1.5) -- (0,1.5) -- (-2,1.5) -- cycle;
\node[cross,thick,minimum size=4pt] at (0.7071,0.7071) {};
\node[cross,thick,minimum size=4pt] at (0.7071,-0.7071) {};
\path (0.8,1.2) node {$x^1_1$};
\path (0.8,-1.2) node {$x^0_1$};
\path (-1,-1) node {$\omega$};
\path (1.3,0) node {$v$};
\draw (1,0) -- (0.9,-0.1);
\draw (1,0) -- (1.1,-0.1);
\draw (0,0) circle (1);
\end{tikzpicture}
\caption{Examples in the case $T\in(0,M_e^*(X^0,X^1))$.}\label{fig:ex (0,T*)}
\end{center}\end{figure}
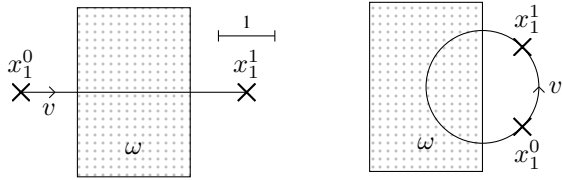

\end{remark}

\section{Algorithm and numerical simulations}\label{sec:num sim}

We consider a crowd described by a discrete configuration $X(t)$ whose evolution is given by System~\eqref{eq ODE}. 
We present the following algorithm to compute numerically the time and the control realizing the exact controllability 
between two configurations satisfying the Geometric Condition \ref{cond1}.

\begin{algorithm}[ht]
\caption{Minimal time problem for exact controllability}
\textbf{Step 1:} Computation of the minimal time  \eqref{OT disc CE}.\\
\textbf{Step 2:} Computation of an optimal permutation to steer $X^0$ to $X^1$ minimizing \eqref{eq:inf}.\\
\textbf{Step 3:} Computation of the control $u$ and the solution $X$ to System \eqref{eq ODE} on $(0,T)$.
\label{algo}
\end{algorithm}

The analysis and convergence of this method for continuous crowds will be studied in the forthcoming paper~\cite{DMR18}.

We now give a numerical example in dimension 2,
for which we solve the minimal time problem with Algorithm \ref{algo}.
Consider $v:=(1,0)$, the control region $\omega$ represented by the rectangle in Figure \ref{fig:simu2D}
and the initial and final configurations $X^0$, $X^1$ given in the first and fourth pictures of Figure \ref{fig:simu2D}.
We control the crowd at time $T=T_e(X^0,X^1)+\delta$, with $\delta=0.1$.

\begin{figure}[ht]\begin{center}
\hspace*{-6mm}\includegraphics[scale=0.5]{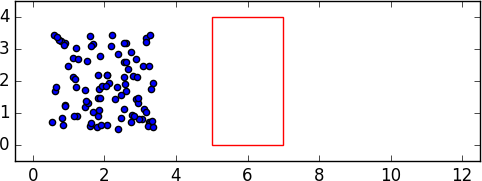}\vspace*{1mm}
\hspace*{-6mm}\includegraphics[scale=0.5]{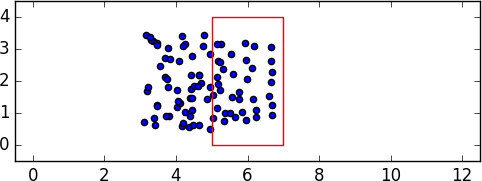}\vspace*{1mm}
\hspace*{-6mm}\includegraphics[scale=0.5]{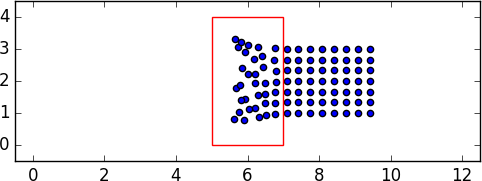}\vspace*{1mm}
\hspace*{-6mm}\includegraphics[scale=0.5]{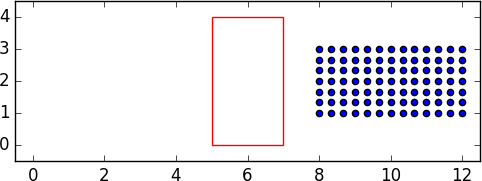}\end{center}
\caption{Solution at time $t=0$, $t=2.58$, $t=5.18$ and $t=T=7.75$.}
\label{fig:simu2D}
\end{figure}


\begin{thebibliography}{99}
\bibitem{axelrod}
R.~Axelrod, \emph{The Evolution of Cooperation: Revised Edition}.\hskip 1em
  plus 0.5em minus 0.4em\relax Basic Books, 2006.

\bibitem{active1}
N.~Bellomo~et al., \emph{Active Particles, Volume 1: Advances in Theory,
  Models, and Applications}, ser. Modeling and Simulation in Science,
  Engineering and Technology.\hskip 1em plus 0.5em minus 0.4em\relax Springer
  International Publishing, 2017.

\bibitem{camazine}
S.~Camazine, \emph{Self-organization in Biological Systems}, ser. Princeton
  studies in complexity.\hskip 1em plus 0.5em minus 0.4em\relax Princeton
  University Press, 2003.

\bibitem{CPTbook}
E.~Cristiani~et al., ``Multiscale modeling of pedestrian dynamics,'' 2014.

\bibitem{helbing}
D.~Helbing and R.~Calek, \emph{Quantitative Sociodynamics: Stochastic Methods
  and Models of Social Interaction Processes}, ser. Theo. and Dec. Lib.
  B.\hskip 1em plus 0.5em minus 0.4em\relax Springer Neth., 2013.

\bibitem{jackson}
M.~Jackson, \emph{Social and Economic Networks}.\hskip 1em plus 0.5em minus
  0.4em\relax Princ. Univ. Press, 2010.

\bibitem{SepulchreReview}
R.~Sepulchre, ``Consensus on nonlinear spaces,'' \emph{Ann. rev. in contr.},
  vol.~35, no.~1, pp. 56--64, 2011.

\bibitem{bullo}
F.~Bullo~et al., \emph{Distributed Control of Robotic Networks: A Mathematical
  Approach to Motion Coordination Algorithms}, ser. Princ. Ser. in Appl.
  Math.\hskip 1em plus 0.5em minus 0.4em\relax Princ. Univ. Press, 2009.

\bibitem{kumar}
V.~Kumar~et al., \emph{Cooperative Control: Block Island Workshop on
  Cooperative Control}, ser. Lecture Notes in Control and Information
  Sciences.\hskip 1em plus 0.5em minus 0.4em\relax Springer Berlin Heidelberg,
  2004.

\bibitem{lin}
L.~Zhiyun~et al., ``Leader--follower formation via complex {L}aplacian,''
  \emph{Automatica}, vol.~49, no.~6, pp. 1900 -- 1906, 2013.

\bibitem{ferscha}
A.~Ferscha and K.~Zia, ``Lifebelt: Crowd evacuation based on vibro-tactile
  guidance,'' \emph{IEEE Pervas. Comp.}, vol.~9, no.~4, pp. 33--42, 2010.

\bibitem{luh}
P.~Luh~et al., ``Modeling and optimization of building emergency evacuation
  considering blocking effects on crowd movement,'' \emph{IEEE Trans. on Autom.
  Sc. and Eng.}, vol.~9, no.~4, pp. 687--700, 2012.

\bibitem{canudas}
C.~{Canudas--de--Wit}~et al., ``Graph constrained-{CTM} observer design for the
  {G}renoble south ring,'' \emph{IFAC Proceedings Volumes}, vol.~45, no.~24,
  pp. 197--202, 2012.

\bibitem{hegyi}
A.~Hegyi~et al., ``Specialist: A dynamic speed limit control algorithm based on
  shock wave theory,'' in \emph{Intel. Transp. Syst.}\hskip 1em plus 0.5em
  minus 0.4em\relax IEEE, 2008, pp. 827--832.

\bibitem{DMR17}
M.~Duprez~et al., ``Approximate and exact controllability of the continuity
  equation with a localized vector field,'' \emph{Submitted}, 2017.

\bibitem{agrabook}
A.~A. Agrachev and Y.~Sachkov, \emph{Control theory from the geometric
  viewpoint}.\hskip 1em plus 0.5em minus 0.4em\relax Springer Science \&
  Business Media, 2013, vol.~87.

\bibitem{jurdjevic}
V.~Jurdjevic, \emph{Geometric control theory}.\hskip 1em plus 0.5em minus
  0.4em\relax Cam. uni. press, 1997, vol.~52.

\bibitem{sontag}
E.~D. Sontag, \emph{Mathematical control theory: deterministic finite
  dimensional systems}.\hskip 1em plus 0.5em minus 0.4em\relax Springer Science
  \& Business Media, 2013, vol.~6.

\bibitem{PRT15}
B.~Piccoli~et al., ``Control to flocking of the kinetic {C}ucker-{S}male
  model,'' \emph{J. Math. Anal.}, vol.~47, no.~6, pp. 4685--4719, 2015.

\bibitem{CPRT17}
M.~Caponigro~et al., ``Mean-field sparse {J}urdjevic-{Q}uinn control,''
  \emph{Math. Models Methods Appl. Sci.}, vol.~27, no.~7, pp. 1223--1253, 2017.

\bibitem{CPRT17b}
------, ``Sparse {J}urdjevic-{Q}uinn stabilization of dissipative systems,''
  \emph{Automatica J. IFAC}, vol.~86, pp. 110--120, 2017.

\bibitem{DMR18}
M.~Duprez~et al., ``Minimal time problem for crowd models with a localized
  vector field,'' \emph{In preparation}, 2018.

\bibitem{V03}
C.~Villani, \emph{Topics in optimal transportation}, ser. Graduate Studies in
  Mathematics.\hskip 1em plus 0.5em minus 0.4em\relax AMS, Providence, RI,
  2003, vol.~58.

\bibitem{C09}
J.-M. Coron, \emph{Control and nonlinearity}, ser. Mathematical Surveys and
  Monographs.\hskip 1em plus 0.5em minus 0.4em\relax Providence, RI: AMS, 2007,
  vol. 136.

\bibitem{KM40}
M.~Krein and D.~Milman, ``On extreme points of regular convex sets,''
  \emph{Studia Math.}, vol.~9, pp. 133--138, 1940.
%
%
%
%
\end{thebibliography}
\end{document}